\documentclass[12pt]{amsart}
\usepackage{graphicx, amsmath, fullpage, amssymb, amsthm, amsfonts, mathrsfs, eucal, moreverb, mathtools, float, bm, caption, subcaption, wrapfig}

\usepackage[p,osf]{cochineal}
\usepackage[scale=.95,type1]{cabin}
\usepackage[zerostyle=c,scaled=.94]{newtxtt}

\usepackage{enumerate}
\usepackage[left=2.3 cm,right=2.3cm,top=1.9cm,bottom=2.4cm]{geometry}

\newtheorem{theorem}{Theorem}[section]
\newtheorem{lemma}[theorem]{Lemma}
\newtheorem{prop}[theorem]{Proposition}
\newtheorem{cor}[theorem]{Corollary}

\theoremstyle{definition}

\usepackage[toc,page]{appendix} 
\usepackage[svgnames]{xcolor}
\definecolor{viridisblue}{HTML}{482878}
\definecolor{viridisgreen}{HTML}{35b779}
\definecolor{viridisgreen2}{HTML}{21918c}
\theoremstyle{remark}
\newtheorem{remark}[theorem]{Remark}
\numberwithin{equation}{section}
\usepackage[colorlinks,
            linkcolor=viridisgreen,
            anchorcolor=viridisgreen2,
            citecolor=viridisblue
            ]{hyperref}
\usepackage{multirow}
\usepackage{todonotes}

\newcommand{\ii}{\mathrm{i}}
\newcommand{\dd}{\mathrm{d}}

\newcommand{\bu}{\mathbf{u}}
\newcommand{\bv}{\mathbf{v}}

\newcommand{\erfc}{\mathrm{erfc}}

\newcommand{\dtv}{\mathrm{d}_{\mathrm{TV}}}
\newcommand{\R}{\mathbb{R}}

\newcommand{\PR}{\mathrm{IPR}}
\newcommand{\IPR}{\mathrm{IPR}}
\newcommand{\Pb}{\mathbf{P}}
\newcommand{\Eb}{\mathbf{E}}
\newcommand{\e}{\mathrm{e}}

\usepackage{tikz}
\usetikzlibrary{shapes.misc}
\usetikzlibrary{shapes.symbols}
\usetikzlibrary{shapes.geometric}
\usetikzlibrary{decorations.markings}
\tikzset{
	dot/.style={circle,fill=black,draw=black,inner sep=.5mm,thick},
	ghostdot/.style={circle,fill=white,draw=black,inner sep=.5mm,thick},
	indot/.style={rectangle,fill=white,draw=black,inner sep=.75mm,thick},
    gluedot/.style={rectangle,fill=white,draw=black,densely dashed,inner sep=1mm,thick},
	outdot/.style={regular polygon, regular polygon sides=3,fill=white,draw=black,inner sep=0.4mm,thick},
	->-/.style={decoration={markings,mark=at position 0.5 with {\arrow{>}}},postaction={decorate}}
}
\definecolor{col1}{HTML}{000000}
\definecolor{col2}{HTML}{EE33FF}
\definecolor{col3}{HTML}{009988}
\definecolor{col4}{HTML}{FF7733}
\definecolor{col5}{HTML}{3399EE}
\definecolor{col6}{HTML}{EE3377}
\definecolor{col7}{HTML}{AAAAAA}

\DeclareMathOperator{\Tr}{\mathrm{Tr}}
\title{The delocalization of eigenvectors of real elliptic matrices}
\author{Lucas Benigni}
\thanks{Lucas Benigni is supported by NSERC RGPIN 2023-03882 \& DGECR 2023-00076 and a FRQNT Etablissement de la Relève Professorale 364387. Guillaume Dubach gratefully acknowledges support from the Fondation de l'Ecole Polytechnique, as well as from the Agence Nationale de la Recherche, ANR-25-CE40-5672 and ANR-25-CE40-1380.}
\address{D\'epartement de Math\'ematiques et Statistique, Universit\'e de Montr\'eal, Montr\'eal (Qu\'ebec), Canada}
\email{lucas.benigni@umontreal.ca}
\author{Simon Coste}
\address{Laboratoire de Probabilités, Statistique et Modélisation, Universit\'e Paris Cit\'e, Paris, France}
\email{simon.coste@u-paris.fr}
\author{Guillaume Dubach}
\address{Centre de Mathématiques Laurent Schwartz, \'Ecole Polytechnique, Palaiseau, France}
\email{guillaume.dubach@polytechnique.edu}
\date{\today}

\begin{document}

\begin{abstract}
    We investigate delocalization phenomena for eigenvectors of real random matrices that are invariant by \emph{orthogonal} transformations. 
    A specific phenomenon with these ensembles is that an eigenvector is typically more localized when its eigenvalue is closer to the real axis -- while for unitarily invariant ensembles, all eigenvectors are delocalized at the same level.
    More precisely, we measure the delocalization level of a vector $x\in \mathbb{C}^N$ using the Inverse Participation Ratio $\PR(x) = N|x|_4^4 / |x|_2^4 \geqslant 1$. A higher IPR means a more localized vector. Using the exact distribution of the Schur decomposition of some paradigmatic rotation-invariant matrix models, we prove that conditionally on having an eigenvalue $\lambda$ with $|\mathfrak{Im}(\lambda)| = y / \sqrt{N}$, the IPR of the associated eigenvector converges in distribution towards a random variable $\ell_y$ with an explicit density depending only on $y$. We then prove that $\ell_y \to 3$ when $y \to 0$ and $\ell_y \to 2$ when $y\to +\infty$, coherently with the observed phenomenon. 

    This result is explicitly proved for higher-order IPRs and for the real Elliptic Ginibre ensemble at every non-symmetry parameter $\tau \in [0,1[$, including the classical real Ginibre ensemble ($\tau=0$). 
\end{abstract}

\maketitle

\begin{figure}[H]
    \centering
    \begin{subfigure}[b]{0.48\textwidth}
        \centering
        \includegraphics[width=\textwidth]{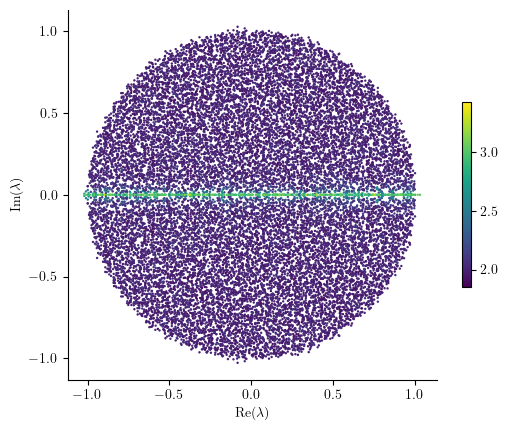}
        \caption{Real Ginibre}
        \label{fig:rge}
    \end{subfigure}
    \begin{subfigure}[b]{0.48\textwidth}
        \centering
        \includegraphics[width=\textwidth]{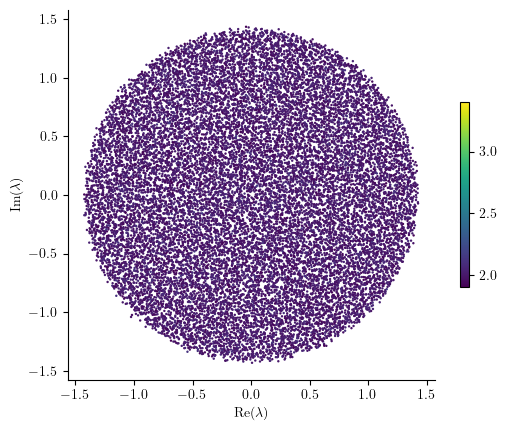}
        \caption{Complex Ginibre}
        \label{fig:cge}
    \end{subfigure}
    \hfill
       \caption{Eigenvalues of 10 realizations of $N\times N$ matrices ($N=2000$) with real entries (left) and complex entries (right). The color of each eigenvalue represents its \emph{inverse participation ratio}, a measure of its delocalization. For the Real Ginibre ensemble, eigenvectors with eigenvalues close to $\mathbb{R}$ are visibly more localized than the others. }
       \label{fig:three graphs}
\end{figure}

\newpage
\section{Introduction}

It is now well understood that the eigenvectors of many models of random matrices are delocalized \cite{erdos2009local, benigni20222optimal, rudelson2015delocalization, optimal_deloc}, in the sense that their total mass is not concentrated in very few indices. 

For non-Hermitian $N\times N$ random matrices with iid entries, the optimal delocalization was proved by Cipolloni and Landon in \cite{cipolloni2023optimal}, improving on the seminal result of Rudelson and Vershynin \cite{rudelson2015delocalization}: any eigenvector $v$ satisfies the bound $|v|_\infty = O(\sqrt{\log N / N})$ with very high probability, under weak assumptions on the distribution of the entries of the matrix. We note that other forms of delocalization exist, for instance no-gaps delocalization which bounds the probability that there are holes within an eigenvector \cite{rudelson2016nogaps, luh2020eigenvector, lytova2020delocalization}.

However, these results do not give a sharper view on \emph{how} exactly the eigenvectors are delocalized across the spectrum. When visualizing the level of delocalization (measured with, say, the $\ell_\infty$ norm, or the kurtosis), it is striking that, for matrices with real entries, the eigenvectors associated with eigenvalues close to the real axis tend to be more localized. Matrices with complex entries seem to have all their eigenvectors delocalized in the same way (which is exactly the case for ensembles with unitary invariance). The difference is particularly visible for Ginibre matrices, as shown in Figure \ref{fig:rge} for real $\mathscr{N}_{\mathbb{R}}(0,1/N)$ entries and in Figure \ref{fig:cge} for complex $\mathscr{N}_{\mathbb{C}}(0,1/N)$ entries. This phenomenon is not restricted to Gaussian matrices, and can be observed in many non-Hermitian models, such as random graphs, charged Ginibre ensembles, or elliptic ensembles, as illustrated in Section \ref{sec:illu}. 

\subsection*{Contribution}In this paper, we give a detailed explanation of this phenomenon (namely, of a relatively higher localization near the real axis) by computing the exact distribution and the limits of the IPR for real Elliptic Ginibre matrices. We expect that these results hold more generally for a variety of orthogonally invariant real random matrix models -- although orthogonal invariance is not enough for this phenomenon to occur. 

\subsection*{Related work}

Eigenvalues and eigenvectors of random matrices have attracted considerable attention from physics and mathematics since the 1950's and particularly in the last twenty years. Most of this attention has been drawn towards Hermitian random matrices until recently, and the eigenvectors of non-Hermitian (or, more specifically, non-normal) matrices remain, to a great extent, quite mysterious. For general (non-integrable) models, the eigenvectors were proved to be delocalized in~\cite{rudelson2015delocalization}, and the optimal delocalization order (namely, $\sqrt{\log N/N}$) was eventually proved in \cite{cipolloni2023optimal}.  

Apart from delocalization, another line of work concerning eigenvectors of various models of non-Hermitian matrices is the study of their Gram matrix (i.e. scalar products interpreted as angles between eigenvectors, studied for instance in \cite{Benaych_Zeitouni}), and the study of the \emph{matrix of overlaps} between eigenvectors, which is the Hadamard product of a Gram matrix of left eigenvectors and a Gram matrix of right eigenvectors. More simply, the entries of this matrix $\mathscr{O}$ can be defined as 
$$
\mathscr{O}_{ij} = \langle L_i | L_j \rangle \ \langle R_j | R_i \rangle
$$
for any bi-orthogonal choice of left eigenvectors $(L_i)_{i=1}^N$ and right eigenvectors $(R_j)_{j=1}^N$.
This object, which is deeply connected to eigenvalue stability, was introduced by Chalker and Mehlig (\cite{chalker1998eigenvector, mehlig2000statistical}) and more recently studied in integrable ensembles \cite{Akemann_determinantal_O, bourgade2020distribution, CrawfordRosenthal, fyodorov2018statistics, WaltersStarr}. There is now a vast literature on eigenvector overlaps; we point to \cite{cipolloni2024optimal, osman2024universality, osman2025bulk, dubova2025gaussian} and references therein for the latest developments and universal results. Most of the works on integrable ensembles consider \emph{complex} random matrices; the phenomena related to \emph{real} non-Hermitian matrices were less studied, mostly due to the fact that the Schur decomposition worked out by Edelman \cite{edelman1997probability} is more intricate. However, a series of papers uses the Schur decomposition in conjunction with supersymmetry to compute the distribution of eigenvector observables such as diagonal overlaps for various models of non-Hermitian random matrices, including real ones: \cite{fyodorov2018statistics, tarnowski2022real, tarnowski2024condition, crumpton2025mean, wurfel2024mean}. It could be enlightening to use these techniques for computing the moments of the eigenvectors, such as the IPR, that is the subject of this paper.

\section{Results}

\subsection{Measuring localization through the IPR}

Our goal is to study the delocalization of eigenvectors of random matrices, measured using the \emph{inverse participation ratio} defined as, for $X\in \mathbb{C}^N$,
\begin{equation}\label{def:ipr}
    \PR_q(X) = \frac{N^{q-1}|X|_{2q}^{2q}}{|X|_2^{2q}}
\end{equation}
where $|\cdot|_p$ denotes the $p$-norm over $\mathbb{C}^n$. By the Cauchy--Schwarz inequality,  $\PR_q(X)=N$ if and only if $X$ has only one non-zero entry, indicating \emph{pure localization}, while $\PR_q(X)=1$ if and only if $X$ is proportional to $\mathbf{1}=(1, \dotsc, 1)$, indicating \emph{pure delocalization}. The quantity $\IPR_2$ is known as the \emph{rescaled kurtosis} and is popular in the physics literature (see \cite{evers2008anderson, evers2000fluctuations} for example). Other measures of delocalization exist, such as the $\ell^\infty$-norm as studied in \cite{rudelson2015delocalization, hoivu, optimal_deloc}.

\subsection{Eigenvectors of invariant ensembles}

Let $G$ be a complex random $N \times N$ matrix, $\lambda \in \mathbb{C}$ be one of its eigenvalues, and $R_\lambda$ an associated unit-norm right eigenvector, by which we mean that 
$$
G R_{\lambda} = \lambda R_{\lambda}, \qquad |R_{\lambda}|_2 = 1.
$$
If we also assume that $G$ is unitarily invariant, i.e. its distribution is invariant by conjugation by any unitary matrix $U$, that is
$$
\forall U \in U_N(\mathbb{C}) \qquad G \stackrel{\mathrm{(d)}}{=} UGU^*, 
$$
then the right eigenvectors are invariant by left multiplication by $U$, meaning that $R_{\lambda}$ is an eigenvector of $G$ for eigenvalue $\lambda$ if and only if $UR_{\lambda}$ is an eigenvector of $UGU^*$ for the same eigenvalue $\lambda$.
It follows that, conditionally on $\lambda$ being an eigenvalue of $G$, $R_{\lambda}$ can be chosen uniformly\footnote{However, note that the eigenvector is defined up to a phase, which one should choose uniformly for this fact to hold.} on the complex unit sphere
$$
\mathbb{S}^{N-1}_{\mathbb{C}} := \{ u = (u_i) \in \mathbb{C}^N \ : \ |u_1|^2 + \dotsb + |u_N|^2 = 1\}.
$$
When the matrix $G$ is \emph{orthogonally} invariant, i.e. 
$$
\forall O \in O_N(\mathbb{R}) \qquad G \stackrel{\mathrm{(d)}}{=} OGO^*, 
$$ then the eigenvectors of $G$ are \emph{not} uniformly distributed over the complex sphere. They can have complex entries with any choice of phase, so in general they also do not belong to the real sphere, 
$$\mathbb{S}^{N-1}_{\mathbb{R}} = \{ u = (u_i) \in \mathbb{R}^N \ : \  u_1^2 + \dotsb + u_N^2 = 1\}.$$
However, the eigenvectors of $G$ associated with a \emph{real} eigenvalue $\lambda \in \mathbb{R}$ can be chosen with a phase that makes them real: for instance, if one chooses a phase that makes the first coefficient real and equally likely to be positive or negative, then the associated eigenvector is automatically uniformly distributed on $\mathbb{S}^{N-1}_{\mathbb{R}}$.

To summarize: eigenvectors of unitarily invariant matrices are (up to a phase) uniformly distributed on $\mathbb{S}^{N-1}_{\mathbb{C}}$; eigenvectors of orthogonally invariant matrices are (up to a sign) uniformly distributed on $\mathbb{S}^{N-1}_{\mathbb{R}}$ \emph{if their eigenvalue is real}; but eigenvectors of orthogonally invariant matrices are, in general, not uniformly distributed on $\mathbb{S}^{N-1}_{\mathbb{C}}$ nor on $\mathbb{S}^{N-1}_{\mathbb{R}}$, and our present goal is to describe the interpolation between these two extreme cases. . 

\bigskip

In this context, it is worth recalling how the localization of a random uniform vector $U$ on the unit sphere depends of the field, $\mathbb{R}$ or $\mathbb{C}$. We recall the double factorial notation,  
$$(2q-1)!! := (2q-1) \times (2q-3) \times \cdots \times 1 = \frac{(2q)!}{2^q q!}. $$
The following elementary result shows that complex vectors are less localized than their real counterparts. 
\begin{lemma}Let $U$ be a uniform vector on $\mathbb{S}^{N-1}_{\Bbbk}$. Then, almost surely, 
\begin{equation}\label{pure_haar}
    \lim_{N\to\infty}\PR_q(U) = \begin{cases}
        {(2q-1)!!} \text{ if } \Bbbk = \mathbb{R} \\[1ex] {q!} \text{ if } \Bbbk = \mathbb{C}. 
    \end{cases}
\end{equation}
For instance, the limit of the rescaled kurtosis $\IPR_2(U)$ is 3 for real vectors and 2 for complex vectors. 
\end{lemma}

\begin{proof}
    By the rotational invariance of Gaussian variables, $U$ has the same distribution as $V / |V|_2$ where $V_i \sim \mathscr{N}_\Bbbk (0,1)$, all of the entries of $V$ being independent. By the law of large numbers, 
    \[ \PR_q\left(
            \frac{V}{|V|_2}
        \right) 
        =
        \frac{\frac{1}{N}\sum_{i=1}^N |V_i|^{2q}}
        {\left(
            \frac{1}{N}\sum_{i=1}^N \vert V_i\vert^2
        \right)^q}
        \to
        \frac{\mathbf{E}[|V_1|^{2q}]}{\mathbf{E}[|V_1|^2]^q} 
        =
        {\mathbf{E}[|V_1|^{2q}]}.
        \]
    If $\xi \sim \mathscr{N}_{\mathbb{R}}(0,1)$, then $\mathbf{E}[|\xi|^{2q}] = (2q-1)!!$; and if $\xi \sim \mathscr{N}_{\mathbb{C}}(0,1)$, then $\mathbf{E}[|\xi|^{2q}] = q!$. 
 To put it differently, one recovers the moments of real (resp. complex) Gaussian variables, because uniform vectors can be sampled from vectors with i.i.d. Gaussian entries.
\end{proof}

It turns out that eigenvectors of the Ginibre or Real Elliptic ensembles will be shown to interpolate between these two cases, depending on the amplitude of the imaginary part of the associated eigenvalue: eigenvectors of real eigenvalues are uniformly distributed on $\mathbb{S}^{N-1}_\mathbb{R}$, and eigenvectors in the ``complex bulk’’ are asymptotically uniformly distributed on the complex sphere $\mathbb{S}_{\mathbb{C}}^{N-1}$. The interpolation between both regimes happens continuously in the \emph{depletion regime}, a band of height $O(1/\sqrt{N})$ around the real axis. In this regime, eigenvectors of an eigenvalue $\lambda$ with $\mathfrak{Im}(\lambda) = y/\sqrt{N}$ are complex unit vectors, but they are distributed non-uniformly on $\mathbb{S}_{\mathbb{C}}^{N-1}$: they concentrate around real elements of $\mathbb{S}_{\mathbb{C}}^{N-1}$ when $y$ is small, and they spread over all $\mathbb{S}_{\mathbb{C}}^{N-1}$ when $y$ is large. 

\subsection{The real Elliptic Ginibre ensemble}

All our results will be exactly proved for an important generalization of the Ginibre models, which is the real Elliptic Ginibre ensemble (REG). We follow the presentation given in \cite[Section 7.9]{Byun2025}. 

Let $S=S^\top$ be a matrix from the Gaussian Orthogonal ensemble i.e $(H_{ij})_{i\geqslant j}$ are independent with $H_{ij}\sim \mathscr{N}_\R\left(0,1+\delta_{ij}\right)$ and let $A=-A^\top$ be an antisymmetric matrix such that $(A_{ij})_{i>j}$ are iid $\mathscr{N}_\R(0,1)$ random variables; then we construct our ensemble as 
\begin{equation}\label{eq:defelliptic}
X_\tau = \frac{1}{\sqrt{2N}}\left(\sqrt{1+\tau}H + \sqrt{1-\tau}A\right),
\end{equation}
which can also be written as $\sqrt{(1+\tau)/2N}(H + \sqrt{c}A)$ where $c = (1-\tau)/(1+\tau)$. The random matrices $S$ and $A$ are invariant by orthogonal conjugation, hence so is the REG. The model interpolates between the real Ginibre ensemble ($\tau = 0$) and the classical real symmetric GOE ($\tau = 1$). Our results will be valid in the whole range of $\tau \in [0,1[$.

\subsection{Limiting IPR distribution for the real Elliptic Ginibre ensemble}

Our results for $\PR_q$ are formulated in terms of 

\begin{enumerate}
    \item[(A)] the Legendre polynomials $(L_q)_{q \in \mathbb{N}}$, which are the orthogonal polynomials associated with the uniform weight $w(x) = \mathbf{1}_{x \in [-1, 1]}$ over the interval $[-1,1]$ and 
    \item[(B)] a family of random variables $(S_y)_{y \in \mathbb{R}_+}$, with $S_y$ being distributed as a $\mathscr{N}(0,(1-\tau^2)/4y^2)$ conditioned on being greater than 1 (see \eqref{eq:density_Sy} for the exact density). 
\end{enumerate}

\begin{theorem}\label{thm:main}
    Let $G$ be an $N\times N$ matrix from the real Elliptic Ginibre ensemble for any $\tau \in [0,1[$, and let $R_\lambda$ be a unit eigenvector associated with an eigenvalue $\lambda$ of $X$.
    \begin{itemize}
        \item (Away from real axis) Conditionally on $\lambda= x+\ii y$ with fixed $ x \in \R , y >0$, 
    $$\PR_q(R_\lambda) \xrightarrow[N\to \infty]{\mathrm{a.s.}} q!.$$
    \item (Real axis) Conditionally on $\lambda= x $ with $ x \in \R $, 
    $$\PR_q(R_\lambda) \xrightarrow[N\to \infty]{\mathrm{a.s.}} (2q-1)!!.$$
    \item (Depletion regime) Conditionally on $\lambda= x + \ii y/\sqrt{N}$ with fixed $ x \in \R , y >0$, we have
    \begin{equation}
        \PR_q(R_\lambda) \xrightarrow[N\to \infty]{\mathrm{(d)}} \ell_{q,y}
    \end{equation}
    where $\ell_{q,y}$ is the positive random variable defined by 
    \begin{equation}\label{def:ell}\ell_{q,y} = q! S_y^{-q}L_q(S_y).\end{equation}
    \end{itemize}
    The family $(\ell_{q,y})_{y \in \mathbb{R}_+}$ interpolates between the bulk regime and the real axis regime, in the sense that for any $y \in ]0,\infty[$ one has $q! < \ell_{q,y} < (2q-1)!!$, and that at the boundaries, 
    \begin{align}
        &\ell_{q,y} \xrightarrow[y \to 0]{\Pb} (2q-1)!!, && \ell_{q,y} \xrightarrow[y \to \infty]{\Pb} q!. 
    \end{align}
\end{theorem}

Note that in this Theorem, we do not condition on anything else than the eigenvalue $\lambda=  x + \ii y$: we do not impose that $\lambda$ lies inside the unit disk. Of course, the probability of $G$ having an eigenvalue at any location at macroscopic distance from the unit disk is exponentially small  in $N$, but conditionally on this event, the distribution of $\PR_q$ and its asymptotics in $N$ do only depend on $y$. 

We found the representation \eqref{def:ell} to be the more versatile; of course, one may use the explicit representations of the Legendre polynomials to work out the exact density of $\ell_{q,y}$ for any $q$ as in Lemma \ref{lem:legendrechangeofvar}. The expressions for small $q$ are easily found out and are shown in the following corollary; the explicit density for the limit of $\PR_2$ might be of independent interest. 

For the pure Ginibre case ($\tau=0$), Figure \ref{fig:density} shows that these distributions start concentrated around 3 when $y \approx0$, and become quickly concentrated around 2 when $y$ grows.

\begin{cor}\label{cor:smalllq}
    For the real Ginibre ensemble ($\tau=0$), the densities of $\ell_{q,y}$ for $q=2,3,4$ are given by
    \[
    \begin{aligned}
    &\gamma_{2,y}(\ell) = \frac{\sqrt{2}\vert y\vert}{\sqrt{\pi}\erfc(\sqrt{2}\vert y\vert)}
    \frac{\e^{-\frac{2y^2}{3-\ell}}}{(3-\ell)^{\frac{3}{2}}}\mathbf{1}_{2<\ell<3}\\[1ex]
    &\gamma_{3,y}(\ell)
    =
    \frac{3\sqrt{2}\vert y\vert}{\sqrt{\pi}\erfc(\sqrt{2}\vert y\vert)}\frac{\e^{-\frac{18y^2}{15-\ell}}}{(15-\ell)^{\frac{3}{2}}}
    \mathbf{1}_{6<\ell<15}\\
    &\gamma_{4,y}(\ell)
    =
    \frac{
        \sqrt{6}\vert y\vert
    }
    {
        2\sqrt{\pi}\erfc(\sqrt{2}\vert y\vert)
    }
    \frac{
    \e^{-\frac{6y^2}{15-\sqrt{120+\ell}}}
    }
    {
    \sqrt{120+\ell}(15-\sqrt{120+\ell})^{\frac{3}{2}} 
    }
    \mathbf{1}_{24<\ell<105}.
    \end{aligned}
    \]
\end{cor}

\begin{figure}[H]
  \centering
  \includegraphics[width=0.3\linewidth]{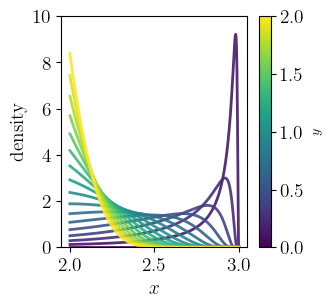}
  \caption{Some densities $\gamma_{2,y}$ from Corollary \ref{cor:smalllq} for various $y$ ranging from $0$ (violet) to $2$ (yellow). When $y$ is large, the distribution converges towards 2, while for small $y$ it converges towards 3. }\label{fig:density}
\end{figure}

\subsection{Extension to other orthogonally-invariant matrix ensembles and universality}\label{subsec:rotinv}
Although our main results are established for the real Elliptic Ginibre Ensemble only, the underlying mechanism naturally should extend to a larger class of real orthogonally invariant matrix models, and also be asymptotically verified in universal cases, i.e. as long as the entries of the matrix are assumed to be independent and under a few natural moments conditions. In this subsection, we explain how the asymptotic law of inverse participation ratios could be investigated beyond the elliptic setting.

A key feature of real non-Hermitian matrices with orthogonal invariance is that, conditionally on the presence of a complex eigenvalue $\lambda$, the matrix can be brought, by an orthogonal similarity transform, into a block upper-triangular form in which the $2\times2$ block associated with $\lambda$ reads
\[
\begin{bmatrix}
x & b \\
-\,c & x
\end{bmatrix},
\]
with $bc>0$ and $b\geqslant c$. The eigenvalues of this block are
\(
\lambda = x \pm i\sqrt{bc},
\)
and the remaining blocks correspond to the rest of the spectrum. This is called the Schur decomposition and it will be further detailed in Section \ref{sec:schur}. At this level, the precise global structure of the matrix plays no role; only the local $2\times2$ geometry associated with the complex eigenvalue matters.

Our results show that, conditionally on observing the eigenvalue
\(
\lambda = x + iy,
\)
with $y=\sqrt{bc}$, the inverse participation ratio of order $q$ of the corresponding eigenvector satisfies the exact identity
\begin{equation}\label{eq:2.5_exact}
\mathbb{E}[\mathrm{IPR}_q \mid \lambda] = q! S^{-q} L_q(S),
\end{equation}
where $L_q$ denotes the Legendre polynomial of order $q$ and where the scale parameter $S$ depends only on the coefficients $b$ and $c$ through
\[
S \;=\; \frac{1}{2}\,\frac{(b+c)^2}{\sqrt{bc}}.
\]
Importantly, \eqref{eq:2.5_exact} is entirely local: it depends neither on the rest of the matrix nor on the global spectral distribution, but only on the parameters entering the $2\times2$ block associated with $\lambda$. We also prove that $\IPR_q$ is exponentially concentrated around this mean value. 

From this perspective, extending the asymptotic distribution of the IPR to other orthogonally invariant ensembles reduces to a single problem: understanding the joint distribution of $(b,c)$ conditionally on the eigenvalue $\lambda$ and notably see if this distribution is $N$-dependent and only depends on $\lambda$. Once this conditional law is identified, the limiting distribution of the IPR follows from the concentration around \eqref{eq:2.5_exact}. 

In the real Elliptic Orthogonal Ensemble (which includes the real Ginibre Ensemble) this conditional distribution can be computed explicitly, which allows for a closed-form characterization of the IPR statistics. However, the argument suggests that other ensembles of interest should be accessible through the same route. In particular, models such as the chiral orthogonal ensemble \cite{akemann2008characteristic}, or more general orthogonally invariant non-Hermitian deformations, may admit tractable descriptions of the conditional law of $(b,c)$, leading to new universality classes for eigenvector localization.
\subsection{Plan of the paper}
\begin{itemize}
\item Section \ref{sec:illu} features illustrations of eigenvector delocalization for various models. 
    \item Section \ref{sec:schur} recalls the real Schur decomposition and its Jacobian. 
    \item Section \ref{sec:density} computes the density of the Schur elements for the real Elliptic Ginibre ensemble. 
    \item Section \ref{sec:proof} uses these distributions to derive an exact, non-asymptotic representation of $\PR_q$ conditioned on a specific eigenvalue, and studies its properties. 
    \item Section \ref{sec:convergence} proves the concentration of the IPR and its almost sure convergence towards its mean.  
    \item Section \ref{sec:proof_main} proves the three points of the main theorem. 
\end{itemize}

\section{Illustrations for various models}\label{sec:illu}

Figure \ref{fig:three graphs} shows the IPR of the Real and Complex Ginibre ensemble and Figure \ref{fig:ell} shows the same for the real Elliptic Ginibre ensemble. 

\begin{figure}
\begin{tabular}{ccc}
     \includegraphics[width=0.3\textwidth]{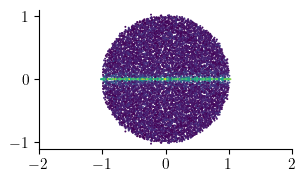}& \includegraphics[width=0.3\textwidth]{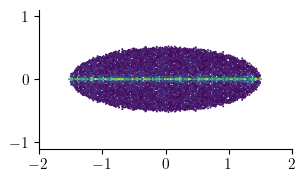}
     & \includegraphics[width=0.3\textwidth]{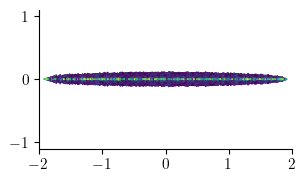}\\
     $\tau=0$ & $\tau=0.5$ & $\tau=0.9$
\end{tabular}
\caption{Three spectra from the real Elliptic Ginibre ensemble at different hermiticity levels $\tau$. The colorbar is the same as in Figure \ref{fig:three graphs}. }\label{fig:ell}
\end{figure}

However, the method we use goes beyond this model or the Elliptic ensemble model, and could in principle apply to any rotation-invariant model--- the only limitation being that we would need to compute the exact distribution of the $2\times 2$ blocks in the real Schur decomposition, as explained in Subsection \ref{subsec:rotinv}. In this section, we illustrate the phenomenon for various models. 

The Induced real Ginibre ensemble with integer charge parameter $\nu$ is the random matrix defined by $G = U \sqrt{X X^\top}$, where $X$ is random $N \times (N + \nu)$ matrix with iid $\mathscr{N}_{\mathbb{R}}(0,1)$ entries, and $U$ is Haar-distributed on $O_N(\mathbb{R})$; this model was studied in \cite{fischmann2012induced}; its exact Schur block density is worked out in \cite{ipsen2015products} and can be used to derive exact results as our main theorem. See top left figure of Figure \ref{fig:othermodels}. 

Sums of Haar-distributed orthogonal matrices are also invariant by rotations, but to our knowledge, there is no closed formula for its Schur decomposition. See top right figure of Figure \ref{fig:othermodels}. 

A discrete version of this model is given by the ``sum-of-permutations'' model. This model is often considered a good proxy for a directed regular graph model, see \cite{coste_rrg}. It is only invariant by row and column permutations, but not by rotations; however, we still observe the same phenomenon where eigenvectors close to the real axis seem to be less delocalized, see bottom left figure of Figure \ref{fig:othermodels}. When the distribution of the permutation matrices is no longer uniform, what happens is not clear. Bottom right figure of Figure \ref{fig:othermodels} shows the IPR of a sum of Ewens-distributed permutations; these random matrices are no longer invariant by row/column permutations. The eigenvectors associated with eigenvalues close to the extremal parts of the pear-shaped spectrum seem to be extremely localized. We have no intuitions on this. 

\begin{figure}[!ht]
    \centering
    \begin{tabular}{cc}
       \includegraphics[width=0.24\textwidth]{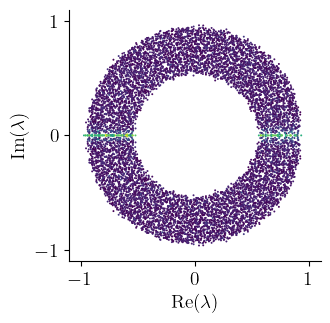} &  \includegraphics[width=0.24\textwidth]{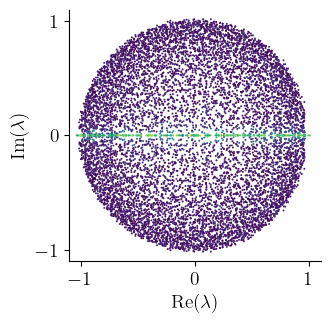}
        \\[-1ex]
       Induced real Ginibre & $H_1+H_2, \, H_i\sim\mathrm{Haar}(O_N(\R))$ \\[1ex]
\includegraphics[width=0.24\textwidth]{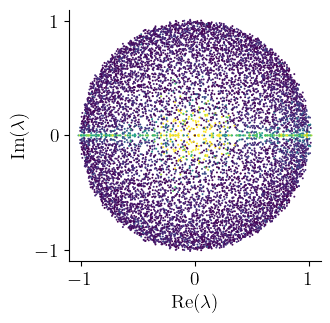}&\includegraphics[width=0.24\textwidth]{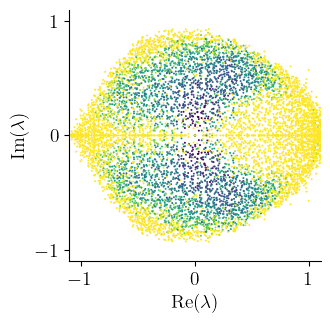} \\[-1ex]
 $P_1+P_2,\,P_i\sim\mathrm{Unif}(S_N)$ & $E_1+E_2,\, E_i \sim \mathrm{Ewens}(S_N)$ 
    \end{tabular}
    \caption{Spectra of other models of non-Hermitian random matrices. The scale for the IPRs is exactly the same as in \ref{fig:rge} (violet is 2 and yellow is 3 or more). Spectra are normalized to be in $D(0,1)$. }
    \label{fig:othermodels}
\end{figure}

\section{The Schur decompositions of real matrices}\label{sec:schur}

The main tool for studying real random matrices invariant by rotations, such as the real Ginibre ensemble, is the real Schur decomposition. In this section, we recall its main properties and we describe the probability distributions of the different elements that will be studied later. Most of those results can be found in Edelman’s seminal paper \cite{edelman1997probability} or in \cite{mehta2004random}. We will start by recalling the real Schur decomposition, and we will then move on to the \emph{incomplete} Schur decomposition which will actually be the one used in our paper for density computations. 

\subsection{The elements of the Schur decomposition}\label{subsec:schur}

The real Schur decomposition expresses any real matrix $G$ as orthogonally similar to a block-triangular matrix $T$, with diagonal blocks of size 2 or 1, i.e $G = OTO^\top$ for some $O \in O_N(\mathbb{R})$; here $T$ is an upper-block-triangular matrix, with $m$ blocks of size $2\times 2$, \begin{align}\label{B}&B_k = \begin{bmatrix}x_k & b_k \\ -c_k & x_k\end{bmatrix}&& b_k c_k>0 \text{  and  } b_k \geqslant c_k \end{align} and $r=N-2m$ blocks of size $1\times 1$, say $B_k=[w_k]$ with $w_k$ real. The $2\times 2$ blocks will sometimes be called ``complex blocks’’. Lemma 4.2 in \cite{edelman1997probability} shows that the shape of $B_k$ in \eqref{B} is essentially unique: any $2\times 2$ matrix with real entries and non-real eigenvalues is orthogonally similar to a unique matrix shaped like \eqref{B}.

The $N=r + 2m$ eigenvalues of $T$ are the real eigenvalues $w_1, \dotsc, w_r$, plus the pairs of complex conjugate eigenvalues of the $m$ blocks of shape \eqref{B}. These eigenvalues are derived in Appendix \ref{app:2x2}:
\begin{align}
    &\lambda_k = x_k + i \sqrt{b_k c_k} && \overline{\lambda_k} = x_k - i \sqrt{b_k c_k}. 
\end{align}

\subsection{The incomplete Schur decomposition}

The \emph{incomplete} real Schur decomposition only keeps ``one block’’ of the real Schur decomposition. We point the reader toward Section 5.2 in Edelman’s paper \cite{edelman1997probability}. This decomposition expresses any real matrix $G$ as $G = OMO^\top$, 

\begin{equation}\label{eq:incomplete_schur}
M=
\begin{bmatrix}
\begin{array}{cc|c}
x   & b   & \multirow{2}{*}{$W$} \\
-c  & x   &                      \\
\hline
\multicolumn{2}{c|}{0} & A_1
\end{array}
\end{bmatrix}
\end{equation}
where, just as before, $bc>0$ and $b\geqslant c$, and $O \in O_N(\mathbb{R})$ must be a product of two Householder reflections. Theorem 5.2 in \cite{edelman1997probability} computed the Jacobian of this change-of-variables, namely
\begin{equation}\label{eq:incomplete_schur_jacobian}
    dG = 2(b-c)\det((A_1 - x I)^2 + y^2) db dc dx dA_1 dO. 
\end{equation}

\subsection{The real eigenvectors of \texorpdfstring{$G$}{G}}\label{subsec:real_eigs}

Let $\mu \in \R$ be a real eigenvalue of $G$, and consider the Schur decomposition $G = OTO^\top$. Without loss of generality, we can suppose that $\mu$ is on the top-left corner of $T$, i.e. 
\begin{align}T = \begin{bmatrix}
    \mu  & * & \dots & *\\ 
    0 & * & \ddots & \vdots \\
    \vdots & \ddots & \ddots & * \\
    0 & \dots & 0 &*
\end{bmatrix}\end{align}
so that $e_1$ is an eigenvector of $T$ with eigenvalue $\mu$. Since $OTO^\top(Oe_1) = \mu Oe_1$, then $R_{\mu} = Oe_1$ is a unit eigenvector of $G = OTO^\top$. This eigenvector is thus real, and uniformly distributed on $S_{\R}^{N-1}$. 

\subsection{The non-real eigenvectors of \texorpdfstring{$G$}{G}}

Consider now the Schur decomposition $G = OTO^\top$. We suppose (without loss of generality) that $T$ has a $2\times 2$ block $A$ on its top-left corner, that is, \begin{align}\label{T}T = \begin{bmatrix}
    x & b & * && *\\ -c & x \\
    0 & & \ddots & &*\\
    &
\end{bmatrix}\end{align}
where $x,b,c$ are real numbers satisfying $bc>0$ and $b\geqslant c$. The spectral decomposition of such $2\times 2$ matrices is worked out in Appendix \ref{app:2x2}; the block $A$ has an eigenvalue in the upper-half complex plane given by
\begin{equation}
    \lambda = x+\ii \sqrt{bc}
\end{equation}
associated with the eigenvector 
$\begin{bmatrix}\ii s \\ t\end{bmatrix}$
where $s,t$ are 
\begin{align}\label{eq:defst}&s = \sqrt{\frac{b}{b+c}}, &&t = \sqrt{\frac{c}{b+c}}, \end{align}
see \eqref{zt} and discussion. Now, the matrix $T$ in \eqref{T} is upper triangular with first block $A$, hence it has an eigenvalue $\lambda$ associated with the unit (right) eigenvector 
\begin{equation}\label{eq:u0}
    R_0 = \begin{bmatrix}
        \ii s \\ t \\ 0 \\ \vdots \\ 0
    \end{bmatrix}.
\end{equation}
Since $OTO^\top(OR_0) = \lambda OR_0$, then $G=OTO^\top$ has a unit eigenvector 
\begin{equation}R = OR_0\end{equation} associated with the eigenvalue $\lambda$. Importantly, this eigenvector has complex entries, since $\ii s$ is itself complex. However, since $\ii s$ is purely complex, the real and imaginary parts of $R$ are found to be
\begin{equation}\label{def:u}
    R = \ii sO_1+tO_2 
\end{equation}
where $O_1 = (O_{k,1}), O_2 = (O_{k,2})$ are the first two columns of the orthonormal matrix $O$.
Since $O$ is Haar-distributed, we can compute the distribution of $R$ conditioned on $A$; indeed, it only depends on $s,t$ which only depend on the off-diagonal parameters $b,c$. Note that since $t^2+s^2=1$, it is clear that $|R|_2 = |R_0|_2 = 1$.

\section{The density of the blocks for the real Elliptic Ginibre ensemble}\label{sec:density}

Let $G$ be a real Elliptic Ginibre matrix and $G = OTO^\top$ its Schur decomposition. By invariance, $O$ is still Haar-distributed over $O_N(\mathbb{R})$. Conditionally on the first block of $T$ to be a $2\times 2$ block just as in \eqref{T}, 
$$\begin{bmatrix}
    x & b \\ -c & x
\end{bmatrix},$$
the representation \eqref{eq:u0} of the eigenvector still holds true. Its distribution depends on the parameters $b,c$ which are in the first block of $T$. In this section, we compute their conditional distribution. 

From the definition \eqref{eq:defelliptic}, one can derive the density of the matrix entries of the REG (see \cite{Byun2025}), 
\begin{equation}\label{eq:density_elliptic}
\frac{1}{Z_{N,\tau}}\exp\left(
-\frac{N}{2(1-\tau^2)}\left(
\Tr(X X^\top) -\tau \Tr(X^2)
\right)
\right)
\end{equation}
with 
\[
Z_{N,\tau}
=
\left(
    2\pi
\right)^{\frac{N^2}{2}}
\left(
\frac{1+\tau}{1-\tau}
\right)^{\frac{N(N-1)}{4}}
\left(
\frac{N}{1+\tau}
\right)^{\frac{N^2}{2}}.
\]

Using \eqref{eq:incomplete_schur_jacobian}, we can compute the distribution of $x,b,$ and $c$ for a given a complex eigenvalue $\lambda=x+\ii \sqrt{bc}$. Our proof follows the ideas in \cite{edelman1997probability}: the joint density of all the elements $(x,b,c, A_1, W, O)$ is given by
\begin{equation}\label{eq:density_elliptic_1}
    \frac{2}{Z_{N,\tau}}(b-c)\det((A_1 - xI_N)^2 + y^2I_N) e^{-\frac{N}{2(1-\tau^2)}(\Tr(XX^\top) - \tau \Tr(X^2)}dbdcdxdA_1dWdO.
\end{equation}
We note that 
\begin{align*}
    \Tr(XX^\top) - \tau \Tr(X^2) &= \Tr(TT^\top) - \tau \Tr(T^2)\\
    &= 2(1-\tau)x^2 + b^2 + c^2 + 2\tau bc + \Tr(WW^\top) + \Tr(A_1A_1^\top)  - \tau \Tr(A_1^2). 
\end{align*}
In \eqref{eq:density_elliptic_1}, integrating out the part which is only related to $W$ and $O$ (these are just constants), we find that the density of $x,b,c$ is proportional to 
\[
 (b-c)\e^{
    -\frac{N}{1+\tau}x^2 - \frac{N}{2(1-\tau^2)}(b^2+c^2+2\tau bc)
}
\underbrace{
    \int_{A_1}\det((A_1-x)^2+bc)
    \e^{
        -\frac{N}{2(1-\tau^2)}(\Tr(A_1 A_1^\top)-\tau\Tr(A_1^2))
    }
    \dd A_1.
    }_{=: F_{N-2, \tau(\lambda).}}
\]
To see why $F_{N-2,\tau}(\lambda)$ only depends on $\lambda$ and not on $x,b,c$, it is enough to note that since $\lambda = x+\ii\sqrt{bc}$, we have $\det((A_1-x)^2+bc) = \det(A_1-\lambda)\det(A_1-\bar{\lambda})$.

\begin{lemma}\label{lemma:rho_y}
    The distribution of $\delta = \sqrt{N}(b-c)$ conditionally on $\lambda = x+\ii y/\sqrt{N}$ has density 
    \[
    \varrho_{y,\tau} = \frac{1}{Z_{y,\tau}}\frac{\delta \e^{-\frac{\delta^2}{2(1-\tau^2)}}}{\sqrt{\delta^2+4y^2}}\mathbf{1}_{\delta>0}
    \quad\text{with}\quad 
    Z_{y,\tau} = \sqrt{\frac{\pi(1-\tau^2)}{2}}\e^{\frac{2y}{1-\tau^2}}\erfc\left(\sqrt{\frac{2}{1-\tau^2}}y\right).
    \]
    In particular, it does not depend on $N$, and the joint distribution of $\sqrt{N}b,$ $\sqrt{N}c$ conditionally on $\lambda=x+\ii y/\sqrt{N}$ also does not depend on $N$.
\end{lemma}
\begin{proof}
The density of $(x,b,c)$ is proportional to 
\[
(b-c)\e^{-\frac{N}{1+\tau}x^2-\frac{N}{2(1-\tau^2)}(b^2+c^2+2\tau bc)}F_{N-2,\tau}(\lambda).
\]
One can change variables from $b,c$ to $y=\sqrt{bc},\delta=\sqrt{N}(b-c)$ (see \cite[Lemma 5.2]{edelman1997probability}): in fact, $b^2+c^2 = (b-c)^2+2\sqrt{bc}^2 = \delta^2/N+2y^2$, hence we obtain that the density of $(x,y,\delta)$ is proportional to
\begin{multline*}
\frac{\delta}{\sqrt{N}}e^{-\frac{N}{1+\tau}x^2-\frac{N}{2(1-\tau^2)}\left(\frac{\delta^2}{N}+2y^2+2\tau y^2\right)}F_{N-2,\tau}(\lambda)\frac{2y}{\sqrt{\delta^2+4Ny^2}}
\\=
\frac{2}{\sqrt{N}}\frac{\delta\e^{-\frac{\delta^2}{2(1-\tau^2)}}}{\sqrt{\delta^2+4Ny^2}}\times y\e^{-\frac{N}{1+\tau}x^2-\frac{N}{1-\tau}y^2}F_{N-2,\tau}(\lambda).
\end{multline*}
The density of $\delta$ conditionally on $\lambda = x+\ii y/\sqrt{N}$ is thus proportional to 
$$\frac{\delta\e^{-\frac{\delta^2}{2(1-\tau^2)}}}{\sqrt{\delta^2+4Ny^2}}. $$
The computation of the normalization constant $Z_{y,\tau} = \int_0^\infty \frac{\delta\e^{-\frac{\delta^2}{2(1-\tau^2)}}}{\sqrt{\delta^2+4Ny^2}}d\delta$ easily leads to the expression given in the Lemma.  
\end{proof}
\section{The Mean of the Inverse Participation Ratio}\label{sec:proof}
We saw in \eqref{def:u} that we can write the right eigenvectors of $G$ as $R_\lambda = tO_2+ \ii s O_1 $
where $s$ and $t$ are given by \eqref{eq:defst} and $O_1, O_2$ are the two first columns of a Haar-distributed orthogonal matrix, and $s,t$ are defined above; in particular, they satisfy $s^2+t^2=1$. 

\subsection{Computation of the moments of \texorpdfstring{$|R_\lambda|$}{the eigenvector}}

Since $R_\lambda$ is a unit vector, we can write
    \[
    \PR_q(R_\lambda)= N^{q-1}|R_\lambda|_{2q}^{2q} = N^{q-1}\sum_{j=0}^N |tO_{2j} + \ii s O_{1j}|^{2q}.
    \]
    By exchangeability of the entries of a Haar distributed orthogonal matrix, we thus have
    \begin{align}
    \Eb\left[\PR_q(R_\lambda)]
    \right]
    &=N^{q-1}\times N
    \Eb\left[
        (tO_2+\ii sO_1)_1^{2q}
    \right]= N^q \Eb[|t^2 O_{21}^2 + s^2 O_{11}^2|^q].\label{eq:exchangeability}
    \end{align}

Since $s,t$ are independent of $O$, we will compute the expectation conditionally on $s$ and $t$. It can be expanded as 
$$\Eb[|R_\lambda|^{2q} \mid s, t] = N^q\sum_{k=0}^q \binom{q}{k}s^{2k}t^{2(q-k)}\Eb[O_{11}^{2k} O_{21}^{2(q-k)}].$$
The joint moments of $O_{11}$ and $O_{21}$ are known, see \cite{meckes2019random}, Proposition 2.5 with $\alpha_1 = 2j, \alpha_2 = 2k$ and $\alpha_3=\dotsb = \alpha_N = 0$: 
\begin{equation}
    \Eb[O_{11}^{2k}O_{21}^{2j}] = \frac{(2k-1)!!(2\ell -1)!!}{N(N+2)\dots (N+2k + 2j- 2)}\qquad k,j\in\mathbb{N}.
\end{equation}

We thus obtain the following expression for the moments of $\PR_q$ at finite size $N$: 
\begin{align}
    \Eb[\PR_q(R_\lambda)\mid s,t] = \frac{N^{q}}{N(N+2)\dots (N+2q- 2)}\sum_{k=0}^q \binom{q}{k}s^{2k}t^{2(q-k)}(2k-1)!!(2(q-k) -1)!!.
\end{align}
It is well known that the moments of standard real Gaussians are given by $\Eb[X^{2k}] = (2k-1)!!$. Introducing two independent random variables $X,Y$ with distribution $\mathscr{N}_{\mathbb{R}}(0,1)$, we can thus write 
$$\sum_{k=0}^q \binom{q}{k}s^{2k}t^{2(q-k)}(2k-1)!!(2(q-k) -1)!! = \Eb[|tX + \ii s Y|^{2q}] $$
which shows the following representation, 
\begin{align}\label{eq:E[IPR|st]}
    \Eb[\PR_q(R_\lambda)\mid s,t] = \frac{N^{q}}{N(N+2)\dots (N+2q- 2)}\Eb[|tX + \ii s Y|^{2q}].
\end{align}

\subsection{A simple representation}

We now further investigate the limit $\Eb[|tX+\ii sY|^{2q}]$ in Proposition \eqref{prop:limit}. This expression can be written in terms of the Legendre polynomials, who can be defined through their generating function (\cite[Chapter 2.2]{jackson2004fourier}): 
\begin{equation}\label{eq:legendre_gf}
\frac{1}{\sqrt{1 - 2xz+z^2}} = \sum_{q=0}^\infty L_q(x)z^q. 
\end{equation}

    \begin{prop}For any $s,t$ such that $st>0$ and $s^2+t^2=1$, 
         \begin{equation}\label{eq:legendre_moments}
    \Eb[|tX+\ii s Y|^{2q}] = q! (2st)^q L_q\left(\frac{1}{2st}\right).
    \end{equation}
    \end{prop}
    \begin{proof}We recall that $\Eb[X^{2k}] = \Eb[Y^{2k}]=(2k)!/(2^k k!)$, hence
    \begin{align*}
    \Eb\left[
        \left\vert
            tX+\ii s Y
        \right\vert^{2q}
    \right]
    =
    \Eb\left[
        (t^2X^2+s^2Y^2)^{q}
    \right]
    &=
    \sum_{k=0}^q
    \binom{q}{k}t^{2k}s^{2(q-k)}\frac{2k!}{2^k k!}\frac{2(q-k)!}{2^{q-k}(q-k)!} \\
    &=\frac{q!}{2^q}
    \sum_{k=0}^q
    t^{2k}s^{2(q-k)} \binom{2k}{k}\binom{2(q-k)}{q-k}.
    \end{align*}
   We recall that the (formal) generating function of the central binomial coefficients is 
   $$g(z)=\sum_{q=0}^\infty \binom{2q}{q}z^q = (1-4z)^{-1/2}.$$
Consequently, for any $z$ we have
\begin{align*}g(s^2z)g(t^2z)=\frac{1}{\sqrt{(1-4s^2z)(1-4t^2z)}} 
    &= \sum_{q=0}^\infty z^q \sum_{k=0}^q \binom{2k}{k}\binom{2(q-k)}{q-k}s^{2k}t^{2(q-k)} \\
    &=\sum_{q=0}^\infty \frac{(2z)^q}{q!} \Eb[tX+\ii sY|^{2q}]
\end{align*}
thus recovering the generating function of the numbers $\Eb[tX+\ii sY|^{2q}]$. But the series development of $((1-4s^2z)(1-4t^2z))^{-1/2}$ is written in terms of the Legendre polynomials. Indeed, 
$$(1-4s^2z)(1-4t^2z) = 1 - 2xw + w^2$$
where
\begin{align*}
    &x = \frac{s^2+t^2}{2\sqrt{s^2t^2}} = \frac{1}{2st}, && w = 4z\sqrt{s^2t^2}=4st z. 
\end{align*}
    We thus obtain
    $$\sum_{q=0}^\infty \frac{(2z)^q}{q!} \Eb[tX+\ii sY|^{2q}] = \sum_{q=0}^\infty 4^q z^q (st)^q L_q\left(\frac{1}{2st}\right)$$
    and by identifying the terms, \eqref{eq:legendre_moments} follows.
\end{proof}

The limit \eqref{eq:legendre_moments} only depends on $S=\frac{1}{2st}$ and is equal to $q! S^{-q}L_q(S)$. We always restricted ourselves to numbers $s,t$ that satisfy $s^2+t^2=1$, hence $S$ is greater than 1; the lemma below studies the behaviour of the limit on $[1, \infty[$. 

\begin{lemma}\label{lem:legendremonot}
    For $x\geqslant1$ and $q\geqslant 2,$ the function $x\mapsto q!x^{-q}L_q(x)$ is strictly increasing and maps $[1,+\infty)$ to $[q!,(2q-1)!!)$.
\end{lemma}

\begin{proof}
  
    We start by the integral representation for $x>1,$ \cite[Chapter 2.11]{jackson2004fourier}
    \[
    L_q(x) = \frac{1}{\pi}\int_0^\pi (x+\sqrt{x^2-1}\cos(\theta))^q \dd x
    \]
    so that we can write 
    \[
    \frac{1}{x^q}L_q(x) = \frac{1}{\pi}\int_0^\pi h(x,\theta)^q\dd x
    \quad\text{with}\quad 
    h(x,\theta) = \frac{1}{x}\left(
        x+\sqrt{x^2-1}\cos(\theta)
    \right)
    \]
    For $x>1$, we can differentiate $h$ and we get 
    \[
    \begin{aligned}
    \partial_x h(x,\theta) 
    &=
    \frac{1}{x}\left(
        1+\frac{x}{\sqrt{x^2-1}}\cos(\theta)
    \right)-\frac{1}{x^2}\left(
        x+\sqrt{x^2-1}\cos(\theta)
    \right)
    \\&=
    \cos(\theta)\sqrt{x^2-1}\left(
        \frac{1}{x^2-1}-\frac{1}{x^2}
    \right).
    \end{aligned}
    \]
    Note that for $x>1$, the sign of $\partial_xh(x,\theta)$ is the same as the sign of $\cos(\theta)$, thus we want to show that 
    \[
    \int_0^\pi \cos(\theta)\left(x+\sqrt{x^2-1}\cos(\theta)\right)^{q-1}\dd\theta >0\quad\text{for }x>1.
    \]
    But we have 
    \[
    \begin{aligned}
        \int_0^\pi \cos(\theta)\left(x+\sqrt{x^2-1}\cos(\theta)\right)^{q-1}\dd\theta
        =
        \sum_{k=0}^{q-1}\binom{q-1}{k}x^{q-1-k}(x^2-1)^{\frac{k}{2}}\int_0^\pi\cos(\theta)^{k+1}\dd \theta 
    \end{aligned}
    \]
    and the result follows since $\int_0^\pi \cos(\theta)^{k+1}\dd \theta$ is zero if $k$ is even and positive if $k$ is odd which gives
    \[
    \frac{\dd}{\dd x}\left(
    \frac{q!}{x^q}L_q(x)
    \right)>0 \quad\text{for }x>1.
    \]
    Besides we have $L_q(1)=1$ so that $(q!x^{-q}L_q(x))_{x=1}=q!$ and the limit at $+\infty$ is given by $q!$ multiplied by the leading coefficient of the Legendre polynomial which is $\frac{1}{2^q}\binom{2q}{q}$; altogether, we get that
    \[
    \lim_{x\to\infty}\frac{q!}{x^q}L_q(x) = \frac{q!}{2^q}\binom{2q}{q} = (2q-1)!!.
    \]
\end{proof}

\section{Convergence of the IPR conditionally on \texorpdfstring{$s$ and $t$}{s and t}}\label{sec:convergence}

\begin{prop}\label{prop:limit} For any fixed real numbers $s,t$, 
    \begin{equation}\label{lim}
        \Eb[\PR_q(R_\lambda)] \xrightarrow[N \to \infty]{}{\Eb[|tX+\ii s Y|^{2q}]}
    \end{equation}
    where $X,Y$ are i.i.d. real standard Gaussian random variables $\mathscr{N}_\mathbb{R}(0, 1)$. 
\end{prop}
\begin{proof}
    One simply has to take $N\to\infty$ in \eqref{eq:E[IPR|st]}. 
        \end{proof}

\begin{remark}\label{rq:diaconis}
    We actually proved that for any non-random, fixed numbers $s,t$ and any Haar-distributed $O\in O_N(\mathbb{R})$, we have 
    $$\Eb\left[\PR_q(tO_2+ \ii s O_1)\right] \to \Eb[|tX + \ii s Y|^{2q}].$$
    
This is essentially a special case of the Diaconis-Freedman convergence theorem (\cite{diaconisfreedman}), where it is shown that 
\[
    \dtv\left(
        \mathrm{Law}(\sqrt{N}O_{11}, \sqrt{N}O_{21})
        -
        \mathrm{Law}(X,Y)
    \right)\xrightarrow[N\to\infty]{} 0.
    \]
However, this last result does not directly apply to our setting, since the total variation distance does only support the convergence of expectations of bounded functions of the random variables, not polynomials. A simple truncation argument could also be used, but we found our presentation to be simpler and more elementary. 
\end{remark}
        
We can now prove \emph{almost sure} convergence of $\IPR_q$ towads its expectation. We use concentration inequalities on the Stiefel manifold, 
\[
\mathscr{V}_2(\R^N) = \left\{
    \mathbf{u}_1,\mathbf{u}_2\in\R^N, \langle \mathbf{u}_i,\mathbf{u}_j\rangle = \delta_{ij} 
\right\}\subset \mathbb{S}^{N-1}\times \mathbb{S}^{N-1}.
\]
The following result is taken from \cite[Theorem 2.4]{ledoux2001concentration}. 
\begin{theorem}\label{thm:concentrationmanif}
    Let $f:\mathscr{V}_2(\R^n)\to \R$ a $L$-Lipschitz function. Then, there are constants $c,C>0$ such that for $(\bu,\bv)$ uniformly distributed on $\mathscr{V}_2(\R^n)$, 
    \[
    \Pb\left(
        \vert f(\bu,\bv)-\Eb[f(\bu,\bv)]\vert \geqslant t
    \right)\leqslant C\e^{-\frac{cNt^2}{L^2}}.
    \]
\end{theorem}

\begin{cor}\label{cor:pscv}
    For any $s,t$ such that $t^2 + s^2 = 1$ we have 
    \[
        \IPR_q(tO_2+\ii sO_1)
    \xrightarrow[n\to\infty]{\mathrm{a.s}}
    \Eb\left[
    \vert tX+\ii s Y\vert^{2q}
    \right]
    \]
    where $X,Y$ are i.i.d. standard real Gaussian random variables $\mathscr{N}_\R(0,1)$.
\end{cor}
\begin{proof}
    The proof is based on computing the Lipschitz constant of $f:\mathscr{V}_2(\mathbb{R}^N)\to\mathbb{R}$ defined by
    \[
    f(\bu,\bv) = \frac{1}{N}\sum_{j=1}^N \left(
        t^2N\bu_j^2+s^2N\bv_j^2
    \right)^q
    =
    N^{q-1}\sum_{j=1}^N \left(
        t^2\bu_j^2+s^2\bv_j^2
    \right)^q. 
    \]
    We will use the fact that 
    \[
    \Vert f\Vert_{\mathrm{Lip}}\leqslant \sup_{\bu,\bv\in \mathscr{V}_2(\R^N)}
    \vert \nabla f(\bu,\bv)\vert_2.
    \]
    Note that $f$ is defined on $\R^N\times \R^N$ and thus we can use the ambient gradient norm. We have 
    \[
    \partial_{\bu_i}f(\bu,\bv)
    =
    N^{q-1}2qt^2\bu_i\left(
        t^2\bu_i^2+s^2\bv_i^2
    \right)^{q-1}
    \]
    and similarly for $\partial_{\bv_i}f(\bu,\bv)$. Thus, using the bounds $\vert \bu_i\vert,\vert \bv_i\vert\leqslant 1$, as well as $s^2+t^2=1$, we get the uniform bound 
    \[
        \Vert f\Vert_{\mathrm{Lip}}\leqslant \sup_{\bu,\bv\in\mathscr{V}_2(\R^N)}\vert \nabla f(\bu,\bv)\vert_2^2 \leqslant 2qN^{q-\frac{1}{2}}.
    \]
    Note that this bound is clearly not strong enough to use Theorem \ref{thm:concentrationmanif}. Indeed, it is a \emph{worst} case bound but not a \emph{typical} one, since $\vert \bu_i\vert^2 \leqslant N^{-1+\varepsilon}$ with overwhelming probability. We define a \emph{good event} by 
    \[
        E_n(\varepsilon) = \left\{
            (\bu,\bv)\in\mathscr{V}_2(\R^N),\vert \bu_i\vert^2\leqslant N^{-1+\varepsilon}, \vert \bv_i\vert^2\leqslant N^{-1+\varepsilon} 
    \right\}
    \]
    then from Theorem \ref{thm:concentrationmanif} with a union bound (which only changes the constant $c>0$ from the bound), we get that there exist $C,c>0$ such that for $N$ large enough,
    \begin{equation}\label{eq:goodevent}
    \Pb\left(
        E_N(\varepsilon)^c
    \right)\leqslant C\e^{-cN^{2\varepsilon}}.
    \end{equation}
    If we consider the restriction of $f$ on $E_n(\varepsilon)$ then the Lipschitz constant can be bounded by
    \[
    \Vert f_{\vert E_n(\varepsilon)}\Vert_{\mathrm{Lip}}\leqslant 
     \sup_{\bu,\bv\in\mathbb{V}_2(\R^N)}\vert \nabla f(\bu,\bv)\vert_2^2
     \leqslant 
     2qN^{\varepsilon(q-\frac{1}{2})}
    \]
    By McShane's theorem \cite[Theorem 6.2]{heinonen2001lectures}, we can extend $f_{\vert E_N(\varepsilon)}$ by a Lipschitz function $\hat{f}$ on $\mathscr{V}_2(\mathbb{R}^N)$ with the same Lipschitz constant as above. Besides, since we can write 
    \[
    \Eb[f]  = \Eb[\hat{f}]+\Eb[(f-\hat{f})\mathbf{1}_{E_N(\varepsilon)^c}]
    \]
    and we have $\vert f-\hat{f}\vert \leqslant CN^q$  by the explicit construction of $\hat{f}$ and the definition of $f$, we get, through \eqref{eq:goodevent},  that for $N$ large enough (depending only on $q$ and $\varepsilon$),
    \[
    \Eb[f] = \Eb[\hat{f}] + O(\e^{-c'N^{2\varepsilon}})
    \]
    for some constant $0<c'<c$.
    If we choose for instance $\varepsilon = \frac{1}{q}$ then from Theorem \ref{thm:concentrationmanif} we get that there exists $N_0(q)$ such that for $N\geqslant N_0(q)$,
    \begin{align*}
    \Pb\left(
        \left\{\left\vert 
            f(O_1,O_2)-\Eb[f(O_1,O_2]
        \right\vert
        \right.\right.&\left.\left.\geqslant N^{-\frac{1}{8q}}
        \right\}\cap E_N\left(\frac{1}{q}\right)
    \right)
    \\&
    \leqslant 
    \Pb\left(
        \left\vert 
            \hat{f}(O_1,O_2)-\Eb\left[\hat{f}(O_1,O_2)\right]
        \right\vert
        \geqslant N^{-\frac{1}{6q}}
    \right)
    \\&
    \leqslant C\e^{-cN^{\frac{1}{6q}}}.
    \end{align*}
    Since, \eqref{eq:goodevent} gives us
    \[
    \Pb\left(
        \left\{\left\vert 
            f(O_1,O_2)-\Eb[f(O_1,O_2]
        \right\vert
        \geqslant N^{-\frac{1}{8q}}
        \right\}\cap E_N\left(\frac{1}{q}\right)^c
    \right)
    \leqslant \Pb\left(E_N\left(\frac{1}{q}\right)^c\right)
    \leqslant C\e^{-cN^{\frac{2}{q}}}
    \]
    and the convergence of the expectation from Proposition \ref{prop:limit}, we get the final almost sure convergence using Borel--Cantelli's lemma.
\end{proof}

\section{Proof of the main theorem}\label{sec:proof_main}

\subsection{Proof of the main theorem: real line}

Let $w$ be a real eigenvalue of $G$; the discussion in Subsection \ref{subsec:real_eigs} shows that its eigenvector can be represented as $Oe_1$ where $O$ is Haar-distributed on $O_n(\mathbb{R})$. It is obvious that $R$ is thus uniformly distributed over the real sphere $\mathbb{S}^{N-1}_{\mathbb{R}}$, and that its IPR converges towards $q!$ as explained in Lemma \ref{pure_haar}.  

\subsection{Proof of the main theorem: depletion regime}

It was pointed in Lemma \ref{lemma:rho_y} that the joint distribution of $\sqrt{N}b, \sqrt{N}c$ conditionally on $\lambda = x + \ii y / \sqrt{N}$ is given by $\varrho_y$, and does not depend on $N$. Since $t$ and $s$ only depend on $b/c$ (see \eqref{eq:defst}), their conditional distribution is also independent of $N$. Consequently, conditionally on $\lambda = x+\ii y/\sqrt{N}$, $\PR_q(R_\lambda)$ has the same distribution as $\PR_q(tO_2 + \ii s O_1)$, and Corollary \ref{cor:pscv} applies and shows that

\begin{equation}\label{eq:conv_ipr}
\PR_q(R_\lambda) \xrightarrow[N\to \infty]{\mathrm{(d)}}q!S_y^{-q}L_q(S_y), 
\end{equation}
where $S_y$ is a random variable with the same distribution as $1/2st$ when $s,t$ are conditioned on $\lambda = x +\ii y$. Let us recall that $\delta = \sqrt{N}(b-c)$ and $y/\sqrt{N} = \sqrt{bc}$; from the definition of $s,t$, one sees that

\[
S_y =\frac{1}{2st} = \frac{1}{2}\sqrt{\frac{(b+c)^2}{bc}} = \frac{1}{2}\sqrt{\frac{(b-c)^2+4bc}{bc}} = \frac{\sqrt{\delta^2+4y^2}}{2\vert y\vert}.
\]
\begin{lemma}\label{lem:Sy}$S_y$ has the distribution of $\mathscr{N}(0,(1-\tau^2)/4y^2)$ conditioned on being greater than 1. 
\end{lemma}

\begin{proof}For simplicity, in the proof we don’t indicate the dependency on $y,\tau$ and simply note, eg, $\varrho$ instead of $\varrho_{y,\tau}$. 
    From Lemma \ref{lemma:rho_y}, we know the conditional density of $\delta$: 
    \[
    \varrho(\delta) = 
    \frac{1}{Z_{\tau,y}}
    \frac{\delta\e^{-\frac{\delta^2}{2(1-\tau^2)}}}{\sqrt{\delta^2+4y^2}}
    \mathbf{1}_{\delta>0}.
    \]
    Note that $\delta = 2\vert y\vert \sqrt{S^2-1}\eqqcolon h^{-1}(S)$ which is a strictly increasing function of $S$, so that by the change of variables formula we get the density $\eta$ of $S$, for $u>1$,
    \begin{align}
    \eta(u) = \varrho(h^{-1}(s))\left\vert \frac{\dd}{\dd s}h^{-1}(s)\right\vert &=
    \frac{2\vert y\vert \e^{\frac{2y^2}{1-\tau^2}}}{Z}\e^{-\frac{4y^2s^2}{2(1-\tau^2)}}\mathbf{1}_{s>1} \\
    &=
    \frac{1}{\sqrt{\frac{\pi}{8 y^2(1-\tau^2)}}\erfc\left(\sqrt{y\frac{2}{1-\tau^2}}\right)}\e^{-\frac{4y^2s^2}{2(1-\tau^2)}}\mathbf{1}_{s>1}.\label{eq:density_Sy}
    \end{align}

    This is exactly the density of $\mathscr{N}(0, \frac{1-\tau^2}{4y^2})$ conditioned on being greater than 1.
\end{proof}

This closes the proof of Theorem \ref{thm:main} in the depletion regime. Moreover, since $\ell_{q,y} = q! S_y^{-q}L_q(S_y)$, we can obtain the conditional density of $\ell_{q,y}$ using a similar change of variables. We only perform the computations for the real Ginibre ensemble, although there is no particular obstruction for $\tau>0$. 

\begin{lemma}\label{lem:legendrechangeofvar}
    When $\tau=0$ (real Ginibre ensemble), the conditional distribution of $\ell_{q,y}$ conditionally on $y$ is given by
    \[
    \gamma_y(\ell)
    =
    \frac{2\sqrt{2}\vert y\vert}{\sqrt{\pi}\erfc(\sqrt{2}y)}
    \frac{\e^{-2y^2g_q^{-1}(\ell)^2}}{\vert \Phi_q(g_q^{-1}(\ell))\vert}
    \mathbf{1}_{q!<\ell<(2q-1)!!}
    \]
    where we defined $g_q^{-1}$ as the inverse of the map $[1,+\infty)\to [q!,(2q-1)!!)$, $x\mapsto q!x^{-q}L_q(x)$ and
    \[
    \Phi_q(x) = \frac{q\cdot q!}{x^{q+1}(1-x^2)}\left(
        xL_{q-1}(x)-L_q(x)
    \right).
    \]
\end{lemma}
\begin{proof}
    We know that for $x>1$, $x\mapsto q!x^{-q}L_q(x)$ is strictly increasing from Lemma \ref{lem:legendremonot} and it takes value in $[q!,(2q-1)!!]$, therefore, for $s\in[q!,(2q-1)!!]$ there exists a smooth inverse $s = g_q^{-1}(\ell)$ and the density is then given by
    \[
    \gamma_y(\ell) = \eta_y(g_q^{-1}(\ell))\left\vert \frac{\dd}{\dd \ell} g_q^{-1}(\ell)\right\vert
    \mathbf{1}_{q!<\ell<(2q-1)!!}
    \]
    and we can compute 
    \[
    \frac{\dd}{\dd s} \left(    
        s^{-q}L_q(s)
    \right)
    =
    \frac{1}{s^q}\left(
       L_q'(s)-\frac{qL_q(s)}{s}
    \right)
    =
    \frac{q}{s^{q+1}(1-s^2)}\left(
        sL_{q-1}(s)-L_q(s)
    \right)
    \eqqcolon \frac{1}{q!}\Phi_q(s)
    \]
    where we used the fact that $(1-x^2)L_q'(x)=q(L_{q-1}(x)-xL_q(x)).$ Finally, we obtain 
    \[
    \gamma_y(\ell) = \frac{\eta_y(g_q^{-1}(\ell))}{\vert\Phi(g_q^{-1}(\ell))\vert}\mathbf{1}_{q!<\ell<(2q-1)!!}
    \]
    which gives the final result.
\end{proof}
While we could not find an explicit formula of $g^{-1}_q$ for every $q$ it is possible to invert the function by hand for small values of $q$ which is given in Corollary \ref{cor:smalllq}.

\begin{proof}[Proof of Corollary \ref{cor:smalllq}]
    We start with the case $q=2$, we see that 
    \[
    \ell = \frac{2}{s^2}L_2(s) =
        3-\frac{1}{s^2}
    \Longleftrightarrow
    s = \frac{1}{\sqrt{3-\ell}} = g_2^{-1}(\ell)
    \quad\text{and}\quad \Phi_2(s) = \frac{2}{s^3}.
    \]
so that
    \[
    \Phi_2(g_2^{-1}(\ell)) = \frac{2}{(3-\ell)^{\frac{3}{2}}}
    \]
    and we finally get the density
    \[
    \rho_{\ell_2,y}(\ell)
    =
    \frac{\sqrt{2}\vert y\vert}{\sqrt{\pi}\erfc(\sqrt{2}\vert y\vert)}\frac{\e^{-\frac{2y^2}{3-\ell}}}{(3-\ell)^{\frac{3}{2}}}\mathbf{1}_{2<\ell<3}.
    \]
    For the case $q=3$, we see that
    \[
    \ell = \frac{6}{s^3}L_3(s) = 15-\frac{9}{s^2} \Longleftrightarrow
    s = \frac{3}{\sqrt{15-\ell}}
    =
    g_3^{-1}(\ell)
    \quad\text{and}\quad 
    \Phi_3(s) = \frac{18}{s^3}
    \]
    so that $\Phi_3(g_3^{-1}(\ell)) = 2(15-\ell)^{\frac{3}{2}}/3$, and we finally get the density
    \[
    \varrho_{\ell_3,y}(\ell)
    =
    \frac{3\sqrt{2}\vert y\vert}{\sqrt{\pi}\erfc(\sqrt{2}\vert y\vert)}\frac{\e^{-\frac{18y^2}{15-\ell}}}{(15-\ell)^{\frac{3}{2}}}
    \mathbf{1}_{6<\ell<15}
    \]
    The case $q=4$ is slightly more complicated but we see that 
    \[
    \ell = \frac{24}{s^4}L_q(s) = 105-\frac{90}{s^2}+\frac{9}{s^4}.
    \]
    Using the substitution $u=\frac{1}{x^2}$ and remembering that we want $\ell=24$ for $s=1$ as well as $s>0$, we obtain that 
    \[
    s = \sqrt{\frac{3}{15-\sqrt{120+\ell}}}=g_4^{-1}(\ell).
    \]
    We can also compute $\Phi_4(s) = \frac{180s^2-36}{s^5}$.     In particular, we see that 
    \[
    \Phi_4(g_4^{-1}(\ell))
    =
    \left(
        \frac{540}{15-\sqrt{120+\ell}}-36
    \right)\left(
    \frac{15-\sqrt{120+\ell}}{3}
    \right)^{\frac{5}{2}} = \frac{4}{\sqrt{3}}\sqrt{120+\ell}\left(
        15-\sqrt{120+\ell}
    \right)^{\frac{3}{2}}
    \]
    and putting everything together, we get the formula from Corollary \ref{cor:smalllq}.
\end{proof}

\subsection{Proof of the main theorem: in the bulk}

By combining \eqref{eq:E[IPR|st]} with the representation \eqref{eq:legendre_moments}, we see that conditionally on $\lambda = x + \ii y = x + \ii \frac{y\sqrt{N}}{\sqrt{N}}$, we have 
\begin{equation}\label{eq:IPR_sqrtN}\Eb[\PR_q(R_\lambda)\mid s,t] = \frac{N^q}{N(N+2)\dots (N+2q - 2)} q! S_{y\sqrt{N}}^{-q}L_q(S_{y\sqrt{N}})\end{equation}
where $S_y = 1 / 2st$ has distribution as in Lemma \ref{lem:Sy}. 

\begin{lemma}$S_y$ converges in distribution towards 1.   \end{lemma}

\begin{proof}If $X\sim \mathscr{N}(0,1)$ then $S_y$ has the same distribution as $X/v$ conditioned on being greater than 1, with $v = 2y/\sqrt{1-\tau^2}$. Consequently, for any $t>1$, 
$\Pb(S_y>t) = \Pb(X>tv)/\Pb(X>v)$. The Gaussian tail estimates show that this is equivalent to $e^{-t^2v^2/2}/t e^{-v^2/2}=e^{-v^2(t^2-1) / 2}/t$ which goes to 0 as $y \to \infty$ since $t^2-1>0$. 
\end{proof}

Since $x^{-q}L_q(x)$ is a bounded function on $[1,\infty[$ and is equal to 1 when $x=1$, we see that when $N\to \infty$, the random variables $S_{y\sqrt{N}}^{-q}L_q(S_{y\sqrt{N}})$ converge in distribution towards $q!$; since the limit is constant, the convergence also holds in probability. Using concentration arguments similar to the ones used in the proof of Corollary \ref{cor:pscv}, we could upgrade to almost sure convergence of $\IPR_q(R_\lambda)$ conditionally on $\lambda = x+\ii y,$ we omit the details.

\bibliographystyle{plain}
\bibliography{bibli.bib}

\appendix

\section{Spectral decomposition of \texorpdfstring{$2\times 2$}{2x2} real matrices}\label{app:2x2}

\subsection{The eigenvalues and eigenvectors of \texorpdfstring{$2\times 2$}{2x2}  matrices}

Let $a,b,c,d$ be real numbers, and let
$$A = \begin{bmatrix}
    a & b \\ c & d
\end{bmatrix}.$$

The following proposition can be checked by elementary means.  

\begin{prop}
    
    The matrix $A$ has eigenvalues $\lambda_\pm = m \pm \sqrt{m^2 - p}$, where $m=(a+d)/2$ and $p= ad - bc$. 
    If $c \neq 0$, the associated right-eigenvectors are 
    $$u_\pm =  \begin{bmatrix}
        \lambda_\pm - d \\ c 
    \end{bmatrix} .$$
\end{prop}

When $m^2-p<0$, the eigenvalues $\lambda_\pm$ are complex conjugates; we will restrict to this case. Let us pick one of those complex eigenvalues, say $\lambda = \lambda_+$, and an associated eigenvector, 
$$\begin{bmatrix}
    \lambda_++\sqrt{m^2 - p} - d \\ c 
\end{bmatrix} = \begin{bmatrix}
    m+\ii \sqrt{ p - m^2} - d \\ c 
\end{bmatrix} .$$
The euclidean norm of this eigenvector is \begin{align}\sqrt{c^2+|\lambda-d|^2}&=\sqrt{c^2+(m-d)^2+\sqrt{p-m^2}^2}\\ &=\sqrt{c^2 + m^2 + d^2 - 2md + p - m^2} \\
&= \sqrt{c^2 + d^2 - 2md + ad - bc} \\
&= \sqrt{c^2 - bc}  \end{align} hence the unit eigenvector is given by
\begin{equation}
    \begin{bmatrix}z \\ t \end{bmatrix}:=\frac{1}{\sqrt{c^2-bc}}\begin{bmatrix}
        m+\ii\sqrt{p-m^2} - d \\ c 
    \end{bmatrix}
\end{equation}
so that $|z|^2+t^2=1$ (note that $t$ is real and $z$ can be complex). For future reference, we note that 
\begin{align}\label{def:z_t_0}
    & z  = \frac{a-d}{2\sqrt{c^2 - bc}} + \ii \sqrt{\frac{p-m^2}{c^2-bc}} && t  =\sqrt{\frac{1}{1-b/c}}.
\end{align}

\subsection{The special case of \texorpdfstring{\eqref{B}}{(2.1)}} When $A$ has the shape given in \eqref{B}, that is 
$$A= \begin{bmatrix}
    x & b \\ -c & x 
\end{bmatrix}$$
with $bc>0$ and $b\geqslant c$, the former analysis gives  $p = x^2 + bc$, $m=x$, $m^2 - p = -bc <0$. There are two complex conjugate eigenvalues $\lambda_\pm = x \pm \ii \sqrt{bc}$ and the eigenvector associated with $\lambda_+$ is 
$$\begin{bmatrix}z \\ t\end{bmatrix}$$
and we directly see from the formula \eqref{def:z_t_0} that $z$ is purely complex, say $z = \ii s$, with $s,t$ given by
\begin{align}\label{zt}
    & s:= \sqrt{\frac{bc}{c^2 + bc}} = \sqrt{\frac{1}{1+c/b}}, && t  =\sqrt{\frac{1}{1 + b/c}}.
\end{align}
For future reference, we note that $s^2 + t^2 = 1$ and that
\begin{equation}\label{tz} ts = \sqrt{\frac{1}{(1 + c/b)(1 + b/c)}} = \frac{\sqrt{bc}}{b+c} = \frac{\mathfrak{Im}(\lambda)}{b+c}. \end{equation}

\vspace*{1cm}

\hrulefill

\end{document}